\documentclass[leqno,12pt,a4paper]{article}
\usepackage[dvips]{graphicx}
\usepackage{amssymb}
\usepackage{amsmath}
\usepackage{mathrsfs}
\usepackage{color}
\usepackage{comment}
\usepackage{theorem}
\usepackage{ascmac}
\newcommand{\pref}[1]{(\ref{#1})}
\newcommand{\be}{\begin{equation}}
\newcommand{\ee}{\end{equation}}

\newcommand{\qed}{{\unskip\nobreak\hfil\penalty50\quad\null\nobreak\hfil
	$\square$\parfillskip0pt\finalhyphendemerits0\par\medskip}}

\newcommand{\vecb}{\mbox{\boldmath $ b $}}
\newcommand{\vecc}{\mbox{\boldmath $ c $}}

\newcommand{\vecf}{\mbox{\boldmath $ f $}}

\newcommand{\vecx}{\mbox{\boldmath $ x $}}

\newcommand{\veckappa}{\mbox{\boldmath $ \kappa $}}
\newcommand{\vecnu}{\mbox{\boldmath $ \nu $}}

\newcommand{\vectau}{\mbox{\boldmath $ \tau $}}

\newtheorem{thm}{Theorem}[section]

\theorembodyfont{\rmfamily}

\newtheorem{proof}{\normalfont\itshape Proof.}

\title{Large-time behavior of the $H^{-m}$-gradient flow of length for closed plane curves
}
\author{Kohei Nakamura
\footnote{{\it E-mail address}:\ k.nakamura.498@ms.saitama-u.ac.jp\ (K. Nakamura)}\\
Saitama University, Japan}
\date{\today}
\begin{document}
\maketitle
\begin{abstract}
We consider the $H^{-m}$-gradient flow of length for closed plane curves.
This flow is a generalization of curve diffusion flow.
We investigate the large-time behavior assuming the global existence of the flow.
Then we show that the evolving curve converges exponentially to a circle.
To do this, we use interpolation inequalities between the deviation of curvature and the 
isoperimetric ratio, recently established by Nagasawa and the author.
\\
\\[0.3cm]
\end{abstract}
\section{Introduction}
\par
Let $ \vecf = ( f_1 , f_2 ) \, : \, \mathbb{R} / L \mathbb{Z} \to \mathbb{R}^2 $ be a function such that $ \mathrm{Im} \vecf $ is a closed plane curve with rotation number $ 1 $ and the variable of $ \vecf $ is the arc-length parameter.
The unit tangent vector is $ \vectau = ( f_1^\prime , f_2^\prime ) $.
Let $ \vecnu = ( - f_2^\prime , f_1^\prime ) $ be the inward unit normal vector,
and let $ \veckappa = \vecf^{ \prime \prime } $ be the curvature vector.
The curvature $ \kappa = \veckappa \cdot \vecnu $ is positive when $ \mathrm{Im} \vecf $ is convex.
Since the curve has rotation number $ 1 $, the deviation of curvature is
\[
	\tilde \kappa
	=
	\kappa - \frac 1L \int_0^L \kappa \, ds
	=
	\kappa - \frac { 2 \pi } L .
\]
\par
In this paper, we consider the flow
\begin{align}
\partial_t \vecf = (-1)^m (\partial_s^{2m}\tilde{\kappa})\vecnu.
\label{mth-curvediffusion}
\end{align}
This is the $H^{-m}$-gradient flow of length. Indeed, for any 
$ \varphi \in H^{-m} $ , we have
\begin{align*}
	\left. \frac d { d \epsilon } L ( \vecf + \epsilon \varphi \vecnu ) \right|_{ \epsilon = 0 }
	= & \
	- \int_0^L \tilde \kappa \varphi \, ds
	=
	- \int_0^L \left\{ ( -1 )^m \partial_s^m \tilde \kappa \right\} \left( \partial_s^{-m} \varphi \right) ds
	\\
	= & \
	- \int_0^L \left[ \partial_s^{-m} \left\{ ( -1 )^m \partial_s^{2m} \tilde \kappa \right\} \right] \left( \partial_s^{-m} \varphi \right) ds
	\\
	= & \
	- \langle
	( -1 )^m \partial_s^{2m} \tilde \kappa
	,
	\varphi
	\rangle_{ H^{-m} }.
\end{align*}
This flow has already been considered when $m=0,1$.
\par
When $m=0$, the flow {\rm \pref{mth-curvediffusion}} is 
\begin{align}
\partial_t \vecf =  \tilde{\kappa}\vecnu.
\end{align}
This flow was first considered by Gage \cite{G}, who proved that a simple closed strictly convex initial curve remains so along the flow, and
the evolving curve converges to a circle. However, 
we can not expect that a global solution exists if initial curve is not convex.
Using numerical analysis, Mayer \cite{M} found an example of a curve which develops singularities in finite time 
under the flow. Very little appears to be known regarding sufficient conditions for global existence.
\par
Nagasawa and the author \cite{NN} considered large-time behavior assuming the global existence of the flow. 
We showed that the evolving curve converges exponentially to a circle 
assuming the global existence by using interpolation inequalities in section 2.
\par
When $m=1$, the flow {\rm \pref{mth-curvediffusion}} is 
\begin{align}
\partial_t \vecf =  -(\partial_s^2\tilde{\kappa})\vecnu.
\label{curvediffusion}
\end{align}
This flow was proposed by Mullins \cite{Mu} and we call it \textit{curve diffusion flow}.
The flow is a fourth-order parabolic partial differential equation. Hence we do not expect 
convexity to be preserved along the flow. Indeed, Giga and Ito \cite{GI} showed the existence  
of a simple closed strictly convex plane curve that becomes
non-convex in finite time under the flow. Also, Escher--Ito \cite{EI} and Chou \cite{C} proved that evolving curves may develop singularities
in finite time even when the initial curve is smooth.
\par
On the other hand, there are some results for large-time behavior. Chou \cite{C} showed that the evolving curve converges exponentially to a circle 
assuming the global existence of the flow. Moreover Elliott--Garcke \cite{EG} and Wheeler \cite{W} 
showed the global existence and investigated the large-time behavior for initial data close to a circle.
\par
In this paper, we would like to investigate the large-time behavior of {\rm \pref{mth-curvediffusion}}
assuming the global existence of the flow. In order to do this, we introduce several inequalities in section 2.
In section 3, by using these inequalities, we prove that the evolving curve converges exponentially to a circle.   
\section{Several inequalities}
\setcounter{equation}{0}
\par
In this section we introduce several inequalities which were established in \cite{NN}; for details of the proofs, see \cite{NN}.
\par
For a non-negative integer $ \ell $,
we set
\[
	I_\ell = L^{ 2 \ell + 1 } \int_0^L | \tilde \kappa^{( \ell )} |^2 ds ,
\]
which is a scale invariant quantity (cf.\
\cite{DKS}).
It is important for the global analysis of evolving curves to estimate $I_\ell$.
We have the Gagliardo-Nirenberg inequalities 
\[
  I_\ell\leq CI_m^{\frac{\ell}{m}}I_0^{1-\frac{\ell}{m}},
\]
where $0\leq \ell\leq m$ and $C$ is a positive constant and independent of $L$.
Such inequalities are very useful but only these are not sufficient to estimate $I_0$ because of the way these inequalities make use of $I_0$.
Hence we need a different type of inequality to estimate $I_\ell$ for $\ell \geq 0$.
\par
We introduce the quantity
\[
	I_{-1} = 1 - \frac { 4 \pi A } { L^2 }
	,
\]
where
the $ A $ is the (signed area) given by
\[
	A = - \frac 12 \int_0^L \vecf \cdot \vecnu \, ds .
\]
$I_{-1}$ is also scale invariant,
and is non-negative by the isoperimetric inequality.  
\par
The following inequalities for $I_0, I_{-1}$ were derived by Nagasawa and the author in \cite{NN}.
\begin{thm}
We have
\[
	8\pi^2 I_{-1} \leq
	I_0  
        \leq I_{-1}^{ \frac 12 }
	\left[ L^3 \int_0^L
	\left\{ \kappa^3 \tilde \kappa + ( \tilde \kappa^\prime )^2 \right\} ds
	\right]^{ \frac 12 }
	.
\]
In both cases, equality holds only in the trivial case $ \tilde \kappa \equiv 0 $.
\label{Theorem1}
\end{thm}
\par
From this inequality, we have the following new interpolation inequalities.
\begin{thm}
Let $ 0 \leq \ell \leq m $.
There exists a positive constant $ C = C ( \ell,m ) $ independent of $ L $ such that
\[
	I_\ell
	\leq
	C \left( I_{-1}^{ \frac { m - \ell } 2 } I_m + I_{-1}^{ \frac { m - \ell } { m+1 } } I_m^{ \frac { \ell +1 } { m+1 } } \right)
\]
holds.
\label{Theorem2}
\end{thm}
\section{Large-time bahavior}
\setcounter{equation}{0}
\par
In this section, we investigate the large-time behavior of {\rm \pref{mth-curvediffusion}}
assuming the 
global existence of the flow.
Since {\rm \pref{mth-curvediffusion}} is a parabolic equation,
$ \vecf $ is smooth for $ t > 0 $ as long as the solution exists.
Hence by shifting the initial time,
we may assume the initial data is smooth. Then we have the following theorem. 
\begin{thm}
Assume that $ \vecf $ is a global solution of {\rm \pref{mth-curvediffusion}} 
such that the initial rotation number is $ 1 $ and the initial {\rm (}signed{\rm )} area is positive.
Then for each $ \ell \in \mathbb{N} \cup \{ -1, 0 \} $,
there exist $ C_\ell > 0 $ and $ \lambda_\ell > 0 $ such that
\[
	I_\ell (t) \leq C_\ell e^{ - \lambda_\ell t } .
\]
\label{Theorem3}
\end{thm}
\begin{proof}
We have
\begin{align*}
	\frac { dL } { dt }
	&=
	- \int_0^L \partial_t \vecf \cdot \veckappa \, ds
	=
	(-1)^{m+1} \int_0^L \left( \partial_s^{ 2m } \tilde \kappa \right) \tilde \kappa \,  ds
	=
	- \int_0^L \left( \partial_s^m \tilde \kappa \right)^2 ds
        =
        - \frac{1}{L^{2m+1}}I_m,\\
\frac { dA } { dt }
	&=
	- \int_0^L \partial_t \vecf \cdot \vecnu \, ds
	=
	(-1)^{m+1} \int_0^L \left( \partial_s^{ 2m } \tilde \kappa \right) ds
	=
	0.
\end{align*}
When $\ell=-1$, we have
\[
\frac d { dt } I_{-1}
	=
	\frac d { dt } \left( - \frac { 4 \pi A } { L^2 } \right)
	=
	\frac { 8 \pi A } { L^3 } \frac { dL } { dt }
	=
	- \frac { 8 \pi A } { L^{ 2 ( m + 2 ) } } I_m
	\leq
	- \lambda_{-1} I_{-1},
\]
where $\lambda_{-1}$ is a positive constant. Hence, the exponential decay of $I_{-1}$ follows.
\par
Next we consider the behavior of $I_0$. Since
\[
	\partial_t \kappa
	=
	( -1 )^m \partial_s^{ 2m + 2 } \tilde \kappa
	+
	( -1 )^m \kappa^2 \partial_s^{ 2m } \tilde \kappa,
\]
we have
\begin{align*}
	\frac d { dt } I_0
	= & \
	\frac { dL } { dt } \int_0^L \tilde \kappa^2 ds
	+
	L \int_0^L 2 \tilde \kappa \partial_t \tilde \kappa ds
	+
	L \int_0^L \tilde \kappa^2 \partial_t ( ds )
	\\
	= & \
	- \frac { I_m } { L^{ 2m+1 } } \frac { I_0 } L
	+
	2L \int_0^L \tilde \kappa
	\left\{
	( -1 )^m \partial_s^{ 2m + 2 } \tilde \kappa
	+
	( -1 )^m \kappa^2 \partial_s^{ 2m } \tilde \kappa
        +\frac{2\pi}{L^2}\frac{dL}{dt}
	\right\}
	ds\\
	& +
	( -1 )^{ m+1 } L\int_0^L
	\tilde \kappa^2 \kappa \partial_s^{ 2m } \tilde \kappa ds
	\\
	= & \
	- \frac { I_0 I_m } { L^{ 2 ( m + 1 ) } }
	-
	2L \int_0^L \left( \partial_s^{ m+1 } \tilde \kappa \right)^2 ds
	+
	( -1 )^m L \int_0^L
	\tilde \kappa \kappa \left( 2\kappa - \tilde \kappa \right)
	\partial_s^{2m } \tilde \kappa \, ds\\
        = & \
        - \frac { I_0 I_m } { L^{ 2 ( m + 1 ) } }
	-
	\frac { 1 } { L^{ 2 ( m + 1 ) } } I_{m+1}
	+
	( -1 )^m L \int_0^L
	\tilde \kappa \left(\tilde{\kappa}+\frac{2\pi}{L}\right) 
        \left\{ 2\left(\tilde{\kappa}+\frac{2\pi}{L}\right) - \tilde \kappa \right\}
	\partial_s^{2m } \tilde \kappa \, ds.
\end{align*}
Hence we find
\begin{align}
	\frac d { dt } I_0
	+
	\frac { I_0 I_m } { L^{ 2 ( m + 1 ) } }
	+
	\frac { 2 I_{ m+1 } } { L^{ 2 ( m + 1 ) } }
	=
	(-1)^m L \int_0^L
	\left(
	\tilde \kappa^3 + \frac { 6 \pi } L \tilde \kappa^2 + \frac { 8 \pi^2 } { L^2 } \tilde \kappa
	\right)
	\partial_s^{2m } \tilde \kappa \, ds
        \label{3.1}
        .
\end{align}
For $ k \in \mathbb{N} \cup \{ 0 \} $ and $ m \in \mathbb{N} $,
let $ P_m^k ( \tilde \kappa ) $ be any linear combination of the type
\[
	P_m^k ( \tilde \kappa )
	=
	\sum_{ i_1 + \cdots + i_m = k } c_{ i_1 , \dots , i_m }
	\partial_s^{ i_1 } \tilde \kappa \cdots \partial_s^{ i_m } \tilde \kappa
\]
with universal,
constant coefficients $ c_{ i_1 , \dots , i_m } $.
Similarly we define $ P_0^k $ as a universal constant.
Then terms on the right-hand side of \pref{3.1} are
\begin{gather*}
	(-1)^m L \int_0^L
	\tilde \kappa^3
	\partial_s^{2m} \tilde \kappa \, ds
	=
	L \int_0^L
	P_3^m ( \tilde \kappa ) \partial_s^m \tilde \kappa \, ds
	,
	\\
	(-1)^m 6 \pi \int_0^L
	\tilde \kappa^2
	\partial_s^{2m} \tilde \kappa \, ds
	=
	6 \pi \int_0^L
	P_2^m ( \tilde \kappa ) \partial_s^m \tilde \kappa \, ds
	,
	\\
	(-1)^m \frac { 8 \pi^2 } L \int_0^L
	\tilde \kappa
	\partial_s^{2m} \tilde \kappa \, ds
	=
	\frac { 8 \pi^2 } L \int_0^L
	\left( \partial_s^m \tilde \kappa \right)^2 ds
	=
	\frac { 8 \pi^2 } { L^{ 2 ( m + 1 ) } } I_m,
\end{gather*}
by use of integration by parts. We set
\[
	J_{k,p}
	=
	\left(
	L^{ ( 1 + k ) p - 1 }
	\int_0^L \left| \partial_s^k \tilde \kappa \right|^p ds
	\right)^{ \frac 1p }.
\]
Then we have 
\[
	J_{ n,p }
	\leq
	C J_{ m,q }^\theta J_{0,r}^{ 1 - \theta }
	,
	\quad
	\theta = \frac { n - \frac 1p + \frac 1r } { m - \frac 1q + \frac 1r }
\]
from the Gagliardo-Nirenberg inequalities. Hence we show, by using Young's inequality,
\begin{align*}
	&
	\left|
	L \int_0^L
	P_3^m ( \tilde \kappa ) \partial_s^m \tilde \kappa \, ds
	\right|
	\\
	& \quad
	\leq
	\frac C { L^{ 2( m+1 ) } }
	\sum_{ \ell = 0 }^m \sum_{ k+j = \ell }
	J_{m,4} J_{k,4} J_{j,4} J_{m-k-j,4}
	\\
	& \quad
	\leq
	\frac C { L^{ 2( m+1 ) } }
	\sum_{ \ell = 0 }^m \sum_{ k+j = \ell }
	J_{ m+1 , 2 }^{ \frac { 4m + 1 } { 4 ( m+1 ) } }
	J_{0,2}^{ \frac 3 { 4( m+1 ) } }
	J_{ m+1 , 2 }^{ \frac { 4k + 1 } { 4 ( m+1 ) } }
	J_{0,2}^{ \frac { 4 ( m - k ) + 3 } { 4( m+1 ) } } 
	J_{ m+1 , 2 }^{ \frac { 4j + 1 } { 4 ( m+1 ) } }
	J_{0,2}^{ \frac { 4 ( m - j ) + 3 } { 4( m+1 ) } } 
	J_{ m+1 , 2 }^{ \frac { 4 ( m - k - j ) + 1 } { 4 ( m+1 ) } }
	J_{0,2}^{ \frac { 4 ( k + j ) + 3 } { 4( m+1 ) } } 
	\\
	& \quad
	\leq
	\frac C { L^{ 2( m+1 ) } }
	J_{ m+1 , 2 }^{ \frac { 2m + 1 } { m+1 } }
	J_{0,2}^{ \frac { 2m + 3 } { m+1 } }
	\\
	& \quad
	=
	\frac C { L^{ 2( m+1 ) } }
	I_{ m+1 }^{ \frac { 2m + 1 } { 2 ( m+1 ) } }
	I_0^{ \frac { 2m + 3 } { 2 ( m+1 ) } }
	\\
	& \quad
	\leq
	\frac C { L^{ 2( m+1 ) } }
	\left(
	\epsilon I_{ m+1 } + C_\epsilon I_0^{ 2m+3 }
	\right)
\end{align*}
for any $\epsilon>0$ and appropriate constant $C_\epsilon$. Similarly we also have
\begin{align*}
	&
	\left|
	6 \pi \int_0^L
	P_2^m ( \tilde \kappa ) \partial_s^m \tilde \kappa \, ds
	\right|
	\\
	& \quad
	\leq
	\frac C { L^{ 2( m+1 ) } }
	\sum_{k=0}^m
	J_{m,3} J_{k,3} J_{m-k,3}
	\\
	& \quad
	\leq
	\frac C { L^{ 2( m+1 ) } }
	\sum_{k=0}^m
	J_{ m+1 , 2 }^{ \frac { 6m + 1 } { 6 ( m+1 ) } }
	J_{ 0,2 }^{ \frac 5 { 6 ( m+1 ) } }
	J_{ m+1 , 2 }^{ \frac { 6k + 1 } { 6 ( m+1 ) } }
	J_{ 0,2 }^{ \frac { 6 ( m - k ) + 5 } { 6 ( m+1 ) } }
	J_{ m+1 , 2 }^{ \frac { 6 ( m - k ) + 1 } { 6 ( m+1 ) } }
	J_{ 0,2 }^{ \frac { 6k + 5 } { 6 ( m+1 ) } }
	\\
	& \quad
	\leq
	\frac C { L^{ 2( m+1 ) } }
	J_{ m+1 , 2 }^{ \frac { 4m + 1 } { 2 ( m+1 ) } }
	J_{ 0,2 }^{ \frac { 2m + 5 } { 2 ( m+1 ) } }
	\\
	& \quad
	=
	\frac C { L^{ 2( m+1 ) } }
	I_{ m+1 }^{ \frac { 4m + 1 } { 4 ( m+1 ) } }
	I_0^{ \frac { 2m + 5 } { 4 ( m+1 ) } }
	\\
	& \quad
	\leq
	\frac C { L^{ 2( m+1 ) } }
	\left(
	\epsilon I_{ m+1 } + C_\epsilon I_0^{ \frac { 2m + 5 } 3 }
	\right).
\end{align*}
Therefore we have
\[
	\frac d { dt } I_0
	+
	\frac { I_0 I_m } { L^{ 2 ( m+1 ) } }
	+
	\frac { 2 I_{ m+1 } } { L^{ 2 ( m+1 ) } }
	\leq
	\frac C { L^{ 2( m+1 ) } }
	\left(
	\epsilon I_{ m+1 } + C_\epsilon I_0^{ 2m + 3 }
	+ C_\epsilon I_0^{ \frac { 2m + 5 } 3 }
	+
	I_m
	\right).
\]
By using Young's inequality and Theorem \ref{Theorem2}, we obtain
\begin{align*}
	I_0^{ \frac { 2m + 5 } 3 }
	= & \
	I_0^{ \frac { 2m + 3 } 3 } I_0^{ \frac 23 }
	\leq
	\epsilon I_0 + C_\epsilon I_0
	\leq
	\epsilon I_0
	+
	C_\epsilon
	\left(
	I_{-1}^{ \frac { m+1 } 2 } I_{ m+1 }
	+
	I_{-1}^{ \frac { m+1 } { m+2 } } I_{ m+1 }^{ \frac 1 { m+2 } }
	\right)
	\\
	\leq & \
	\epsilon I_0
	+
	C_\epsilon
	\left(
	I_{-1}^{ \frac { m+1 } 2 } + \epsilon^\prime
	\right)
	I_{ m+1 }
	+
	C_{ \epsilon , \epsilon^\prime } I_{-1}
\end{align*}
where $\epsilon , \epsilon^\prime>0$ and $C_\epsilon$ and $C_{ \epsilon , \epsilon^\prime }$ are appropriate
constants. Similarly, for $m \geq 1$, we have
\begin{align*}
	I_m
	\leq 
	C
	\left(
	I_{-1}^{ \frac 12 } I_{ m+1 }
	+
	I_{-1}^{ \frac 1 { m+2 } } I_{ m+1 }^{ \frac { m+1 } { m+2 } }
	\right)
	\leq
	C
	\left(
	I_{-1}^{ \frac 12 } + \epsilon
	\right)
	I_{ m+1 }
	+
	C_{ \epsilon } I_{-1}.
\end{align*}
Taking $ \epsilon $,
$ \epsilon^\prime $ sufficiently small,
we have
\begin{align}
	\frac d { dt } I_0
	+
	\frac { I_0 I_m } { L^{ 2 ( m+1 ) } }
	+
	\frac { C_1 I_{ m+1 } } { L^{ 2 ( m+1 ) } }
	\leq
	\frac { C_2 } { L^{ 2( m+1 ) } } I_0^{ 2m+3 }
	+
	\frac { C_3 } { L^{ 2( m+1 ) } } e^{ - \lambda_{-1} t }
        \label{3.2}
\end{align}
for sufficiently large $t$. Since
\[
	\frac { dL } { dt }
	+
	\frac { I_m } { L^{ 2m+1 } } = 0,
\]
we have
\[
	\int_0^\infty I_m dt < \infty.
\]
From Wirtinger's inequlity, we obtain
\[
	\int_0^\infty I_\ell dt < \infty
\]
for $ \ell \in \{ 0 $,$ \ldots $,$ m \} $.
From \pref{3.2} and $ \int_0^\infty I_0 dt < \infty $, we can show
\[
	\frac { I_0 I_m } { L^{ 2 ( m+1 ) } }
	>
	\frac { C_2 } { L^{ 2( m+1 ) } } I_0^{ 2m+3 }
\]
for sufficiently large $t$. Hence we have
\[
	I_0 \leq C_0 e^{ - \lambda_0 t }.
\]
\par
Next we consider the behavior of $I_{\ell}$ for $\ell \geq 1$. By direct calculations, we have
\[
	\frac d { dt } I_\ell
	=
	( 2 \ell + 1 ) L^{ 2 \ell } \frac { dL } { dt }
	\int_0^L \left( \partial_s^\ell \tilde \kappa \right)^2 ds
	+
	L^{ 2\ell + 1 }
	\int_0^L \left( \partial_s^\ell \tilde \kappa \right)^2
	\partial_t ( ds )
	+
	2 L^{ 2 \ell + 1 } \int_0^L
	\left( \partial_s^\ell \tilde \kappa \right)
	\left( \partial_t \partial_s^\ell \tilde \kappa \right)
	ds
\]
and 
\begin{align*}
	( 2 \ell + 1 ) L^{ 2 \ell } \frac { dL } { dt }
	\int_0^L \left( \partial_s^\ell \tilde \kappa \right)^2 ds
	= & \
	- ( 2 \ell + 1 ) L^{ 2 \ell }
	\int_0^L \left( \partial_s^m \tilde \kappa \right)^2 ds
	\int_0^L \left( \partial_s^\ell \tilde \kappa \right)^2 ds
	\\
	= & \
	- \frac { 2 \ell + 1 } { L^{ 2( m+1 ) } }
	I_m I_\ell
	,
	\\
	L^{ 2\ell + 1 }
	\int_0^L \left( \partial_s^\ell \tilde \kappa \right)^2
	\partial_t ( ds )
	= & \
	( -1 )^{ m+1 }
	L^{ 2\ell + 1 }
	\int_0^L \left( \partial_s^\ell \tilde \kappa \right)^2
	\kappa \partial_s^{ 2m } \tilde \kappa \, ds
	\\
	= & \
	( -1 )^m
	L^{ 2\ell + 1 }
	\int_0^L \left( \partial_s^\ell \tilde \kappa \right)
	\sum_{ n=0 }^1
        L^{-(1-n)}
	P_{ n+2 }^{ 2m + \ell } ( \tilde \kappa ) \, ds.
\end{align*}
We can show 
\begin{align}
	\partial_t \partial_s^\ell \tilde \kappa
	=
	( -1 )^m \partial_s^{ 2m + \ell + 2 } \tilde \kappa
	+
	( -1 )^m \sum_{ n=0 }^2 L^{ - ( 2 - n ) } P_{ n+1 }^{ 2m + \ell } ( \tilde \kappa )
        \label{3.3}
\end{align}
by induction on $\ell$. Indeed, since 
\[
	\partial_t \kappa
	=
	( -1 )^m \partial_s^{ 2m + 2 } \tilde \kappa
	+
	( -1 )^m \kappa^2 \partial_s^{ 2m } \tilde \kappa,
\]
we have
\begin{align*}
	\partial_t \partial_s \tilde{\kappa}
	= & \
        \partial_t \partial_s \kappa
        =
	\partial_s \partial_t \kappa
	+
	( \partial_t \vecf \cdot \veckappa ) \partial_s \kappa
	\\
	= & \
	\partial_s
	\left\{
	( -1 )^m \partial_s^{ 2m + 2 } \tilde \kappa
	+
	( -1 )^m \kappa^2 \partial_s^{ 2m } \tilde \kappa
	\right\}
	+
	( -1 )^m \kappa \left( \partial_s^{ 2m } \tilde \kappa \right) \left( \partial_s \tilde \kappa \right)
	\\
	= & \
	( -1 )^m \partial_s^{ 2m + 3 } \tilde \kappa
	+
	( -1 )^m 2 \kappa \left( \partial_s \kappa \right) \left( \partial_s^{ 2m } \tilde \kappa \right)
	+
	( -1 )^m \kappa^2 \partial_s^{ 2m + 1 } \tilde \kappa
	+
	( -1 )^m \kappa \left( \partial_s^{ 2m } \tilde \kappa \right) \left( \partial_s \tilde \kappa \right)
	\\
	= & \
	( -1 )^m \partial_s^{ 2m + 3 } \tilde \kappa
	+
	( -1 )^m \sum_{ n=0 }^2 L^{ 2 - n } P_{ n+1 }^{ 2m + 1 } ( \tilde \kappa ).
\end{align*}
Hence we obtain \pref{3.3} when $\ell=1$. If \pref{3.3} holds for $\ell \geq 1$, since
\begin{align*}
	\partial_t \partial_s^{ \ell + 1 } \tilde \kappa
	= & \
	\partial_s \partial_t \partial_s^\ell \tilde \kappa
	+
	( \partial_t \vecf \cdot \veckappa ) \partial_s^\ell \tilde \kappa
	\\
	= & \
	\partial_s \left\{
	( -1 )^m \partial_s^{ 2m + \ell + 2 } \tilde \kappa
	+
	( -1 )^m \sum_{ n=0 }^2 L^{ - ( 2 - n ) } P_{ n+1 }^{ 2m + \ell } ( \tilde \kappa )
	\right\}
	+
	( -1 )^m \kappa \left( \partial_s^{ 2m } \tilde \kappa \right)
	\partial_s \left( \partial_s^\ell \tilde \kappa \right)
	\\
	= & \
	( -1 )^m \partial_s^{ 2m + \ell + 3 } \tilde \kappa
	+
	( -1 )^m \sum_{ n=0 }^2 L^{ - ( 2 - n ) } P_{ n+1 }^{ 2m + \ell + 1 } ( \tilde \kappa ),
\end{align*}
we show \pref{3.3} for $\ell+1$. Hence we have
\begin{align*}
	&
	2 L^{ 2 \ell + 1 } \int_0^L
	\left( \partial_s^\ell \tilde \kappa \right)
	\left( \partial_t \partial_s^\ell \tilde \kappa \right)
	ds
	\\
	& \quad
	=
	2 L^{ 2 \ell + 1 } \int_0^L
	\left( \partial_s^{ \ell + m + 1 } \tilde \kappa \right)^2 ds
	+
	( -1 )^m 2 L^{ 2 \ell + 1 }
	\int_0^L
	\left( \partial_s^\ell \tilde \kappa \right)
	\sum_{ n=0 }^2
	L^{ - ( 2 - n ) } P_{ n+1 }^{ 2m + \ell } ( \tilde \kappa ) \, ds
	\\
	& \quad
	=
	- \frac 2 { L^{ 2 ( m + 1 ) } } I_{ m + \ell + 1 }
	+
	( -1 )^m 2 L^{ 2 \ell + 1 }
	\int_0^L
	\left( \partial_s^\ell \tilde \kappa \right)
	\sum_{ n=0 }^2
	L^{ - ( 2 - n ) } P_{ n+1 }^{ 2m + \ell } ( \tilde \kappa ) \, ds.
\end{align*}
Therefore we have
\[
	\frac d { dt } I_\ell
	+
	\frac { 2 \ell + 1 } { L^{ 2 ( m + 1 ) } } I_m I_\ell
	+
	\frac 2 { L^{ 2 ( m + 1 ) } } I_{ m + \ell + 1 }
	=
	( -1 )^m 2 L^{ 2 \ell + 1 }
	\int_0^L
	\left( \partial_s^\ell \tilde \kappa \right)
	\sum_{ n=0 }^2
	L^{ - ( 2 - n ) } P_{ n+1 }^{ 2m + \ell } ( \tilde \kappa ) \, ds.
\]
When $n=0$, after integration by parts $m$ times, using Theorem \ref{Theorem2} and Young's inequality, 
we have
\begin{align*}
	&
	( -1 )^m 2 L^{ 2 \ell + 1 }
	\int_0^L
	\left( \partial_s^\ell \tilde \kappa \right)
	L^{ - 2 } P_1^{ 2m + \ell } ( \tilde \kappa ) \, ds
	\\
	& \quad
	=
	c L^{ 2 \ell - 1 }
	\int_0^L
	\left( \partial_s^{ m + \ell } \tilde \kappa \right)^2 ds
	=
	\frac c { L^{ 2 ( m + 1 ) } } I_{ m + \ell }
	\\
	& \quad
	\leq
	\frac C { L^{ 2 ( m + 1 ) } }
	\left(
	I_{-1}^{ \frac 12 } I_{ m + \ell + 1 }
	+
	I_{-1}^{ \frac 1 { m + \ell + 2 } } I_{ m + \ell + 1 }^{ \frac { m + \ell + 1 } { m + \ell + 2 } }
	\right)
	\\
	& \quad
	\leq
	\frac C { L^{ 2 ( m + 1 ) } }
	\left\{
	\left( I_{-1}^{ \frac 12 } + \epsilon \right) I_{ m + \ell + 1 }
	+
	C_\epsilon I_{-1}
	\right\}.
\end{align*}
When $n=1$, we have
\[
	( -1 )^m 2 L^{ 2 \ell + 1 }
	\int_0^L
	\left( \partial_s^\ell \tilde \kappa \right)
	L^{ - 1 } P_2^{ 2m + \ell } ( \tilde \kappa ) \, ds
	=
	( -1 )^m 2 L^{ 2 \ell }
	\int_0^L
	\sum_{ k=0 }^{ 2m + \ell }
	c_k
	\left( \partial_s^\ell \tilde \kappa \right)
	\left( \partial_s^k \tilde \kappa \right)
	\left( \partial_s^{ 2m + \ell - k } \tilde \kappa \right)
	ds.
\]
We set
\begin{align*}
	K_1 = & \
	\left\{ k \in \{ 0 , \, \ldots , 2m + \ell \} \, | \, \max\{ k , 2m + \ell - k \} > m + \ell \right\}
	,
	\\
	K_2 = & \
	\left\{ k \in \{ 0 , \, \ldots , 2m + \ell \} \, | \, \max\{ k , 2m + \ell - k \} \leq m + \ell \right\}.
\end{align*}
If $\max\{ k , 2m + \ell - k \} > m + \ell$, then $\min\{ k , 2m + \ell - k \} < m + \ell$.
When  $ k \in K_1 $, from integration by parts $ \max\{ k , 2m + \ell - k \} - m - \ell $ times, we have
\begin{align*}
	&
	( -1 )^m 2 L^{ 2 \ell }
	\int_0^L
	\sum_{ k=0 }^{ 2m + \ell }
	c_k
	\left( \partial_s^\ell \tilde \kappa \right)
	\left( \partial_s^k \tilde \kappa \right)
	\left( \partial_s^{ 2m + \ell - k } \tilde \kappa \right)
	ds
	\\
	& \quad
	=
	2 L^{ 2 \ell } \int_0^L
	\sum_{ k \in K_1 }
	( -1 )^{ \max\{ k , 2m + \ell - k \} - \ell }
	c_k \left( \partial_s^{ m + \ell } \tilde \kappa \right)
	P_2^{ m + \ell } ( \tilde \kappa )
	\, ds
	\\
	& \quad \qquad
	+ \,
	( -1 )^m 2 L^{ 2 \ell }
	\int_0^L
	\sum_{ k \in K_2 }
	c_k
	\left( \partial_s^\ell \tilde \kappa \right)
	\left( \partial_s^k \tilde \kappa \right)
	\left( \partial_s^{ 2m + \ell - k } \tilde \kappa \right)
	ds
	\\
	& \quad
	\leq
	\frac C { L^{ 2 ( m + 1 ) } }
	\sum_{ k \in K_1 } \sum_{ k^\prime = 0 }^{ m + \ell }
	J_{ k^\prime , 3 } J_{ m + \ell - k^\prime , 3 } J_{ m + \ell , 3 }
	+
	\frac C { L^{ 2 ( m + 1 ) } }
	\sum_{ k \in K_2 }
	J_{ \ell , 3 } J_{ k , 3 } J_{ 2m + \ell - k , 3 }
	\\
	& \quad
	\leq
	\frac C { L^{ 2 ( m + 1 ) } }
	\sum_{ k \in K_1 } \sum_{ k^\prime = 0 }^{ m + \ell }
	J_{ m + \ell + 1 , 2 }^{ \frac { 6 k^\prime + 1 } { 6 ( m + \ell + 1 ) } }
	J_{ 0,2 }^{ \frac { 6 ( m + \ell - k^\prime ) + 5 } { 6 ( m + \ell + 1 ) } }
	J_{ m + \ell + 1 , 2 }^{ \frac { 6 ( m + \ell - k^\prime ) + 1 } { 6 ( m + \ell + 1 ) } }
	J_{ 0,2 }^{ \frac { 6 k^\prime + 5 } { 6 ( m + \ell + 1 ) } }
	J_{ m + \ell + 1 , 2 }^{ \frac { 6 ( m + \ell ) + 1 } { 6 ( m + \ell + 1 ) } }
	J_{ 0,2 }^{ \frac 5 { 6 ( m + \ell + 1 ) } }
	\\
	& \quad \qquad
	+ \,
	\frac C { L^{ 2 ( m + 1 ) } }
	\sum_{ k \in K_2 }
	J_{ m + \ell + 1 , 2 }^{ \frac { 6 \ell + 1 } { 6 ( m + \ell + 1 ) } }
	J_{ 0,2 }^{ \frac { 6 m + 5 } { 6 ( m + \ell + 1 ) } }
	J_{ m + \ell + 1 , 2 }^{ \frac { 6 k + 1 } { 6 ( m + \ell + 1 ) } }
	J_{ 0,2 }^{ \frac { 6 ( m + \ell - k ) + 5 } { 6 ( m + \ell + 1 ) } }
	J_{ m + \ell + 1 , 2 }^{ \frac { 6 ( 2m + \ell - k ) + 1 } { 6 ( m + \ell + 1 ) } }
	J_{ 0,2 }^{ \frac { - 6 ( m - k ) + 5 } { 6 ( m + \ell + 1 ) } }
	\\
	& \quad
	\leq
	\frac C { L^{ 2 ( m + 1 ) } }
	J_{ m + \ell + 1 , 2 }^{ \frac { 4 ( m + \ell ) + 1 } { 2 ( m + \ell + 1 ) } }
	J_{ 0,2 }^{ \frac { 2( m + \ell ) + 5 } { 2 ( m + \ell + 1 ) } }
	\\
	& \quad
	=
	\frac C { L^{ 2 ( m + 1 ) } }
	I_{ m + \ell + 1 }^{ \frac { 4 ( m + \ell ) + 1 } { 4 ( m + \ell + 1 ) } }
	I_0^{ \frac { 2( m + \ell ) + 5 } { 4 ( m + \ell + 1 ) } }
	\\
	& \quad
	\leq
	\frac C { L^{ 2 ( m + 1 ) } }
	\left(
	\epsilon I_{ m + \ell + 1 }
	+
	C_\epsilon I_0^{ \frac { 2 ( m + \ell ) + 5 } 3 }
	\right).
\end{align*}
When $ n = 2 $, we have
\[
	P_3^{ 2m + \ell } ( \tilde \kappa )
	=
	\sum_{ \alpha = 0 }^{ 2m + \ell }
	\sum_{ k + j = \alpha }
	c_{k,j}
	\left( \partial_s^k \tilde \kappa \right)
	\left( \partial_s^j \tilde \kappa \right)
	\left( \partial_s^{ 2m + \ell - k - j }\tilde \kappa \right).
\]
We set 
\begin{align*}
	K_{ \alpha , 1 }
	= & \
	\left\{
	( k,j ) \, | \,
	k + j = \alpha , \,
	\max \{ k, j, 2m + \ell - k - j \} > m + \ell
	\right\}
	,
	\\
	K_{ \alpha , 2 }
	= & \
	\left\{
	( k,j ) \, | \,
	k + j = \alpha , \,
	\max \{ k, j, 2m + \ell - k - j \} \leq m + \ell
	\right\}.
\end{align*}
If $\max \{ k, j, 2m + \ell - k - j \} > m + \ell$, the other terms are less than $m + \ell$.
When  $ k \in K_{ \alpha , 1 } $, from integration by parts $ \max \{ k, j, 2m + \ell - k - j \} - m - \ell $ times, we have
\begin{align*}
	&
	( -1 )^m 2 L^{ 2 \ell +1 }
	\int_0^L
	\sum_{ \alpha = 0 }^{ 2m + \ell }
	\sum_{ k + j = \alpha }
	c_{kj}
	\left( \partial_s^k \tilde \kappa \right)
	\left( \partial_s^j \tilde \kappa \right)
	\left( \partial_s^{ 2m + \ell - k - j }\tilde \kappa \right)
	ds
	\\
	& \quad
	=
	2 L^{ 2 \ell +1 }
	\int_0^L
	\sum_{ \alpha = 0 }^{ 2m + \ell }
	\sum_{ ( k,j ) \in K_{ \alpha , 1 } }
	( -1 )^{ \max \{ k, j, 2m + \ell - k - j \} - \ell }
	c_{kj} \left( \partial_s^{ m + \ell } \tilde \kappa \right)
	P_3^{ m+ \ell } ( \tilde \kappa )
	\, ds
	\\
	& \quad \qquad
	+ \,
	( -1 )^m 2 L^{ 2 \ell +1 }
	\int_0^L
	\sum_{ \alpha = 0 }^{ 2m + \ell }
	\sum_{ ( k,j ) \in K_{ \alpha , 2 } }
	\left( \partial_s^k \tilde \kappa \right)
	\left( \partial_s^j \tilde \kappa \right)
	\left( \partial_s^{ 2m + \ell - k - j }\tilde \kappa \right)
	ds
	\\
	& \quad
	\leq
	\frac C { L^{ 2 ( m + 1 ) } }
	\sum_{ \alpha = 0 }^{ 2m + \ell }
	\sum_{ ( k,j ) \in K_{ \alpha , 1 } }
	\sum_{ \beta = 0 }^{ m + \ell }
	\sum_{ k^\prime + j^\prime = \beta }
	J_{ k^\prime , 4 } J_{ j^\prime , 4 }
	J_{ m + \ell - k^\prime j^\prime , 4 }
	J_{ m + \ell , 4 }
	\\
	& \quad \qquad
	+ \,
	\sum_{ \alpha = 0 }^{ 2m + \ell }
	\sum_{ ( k,j ) \in K_{ \alpha , 2 } }
	J_{ \ell , 4 } J_{ k,4 } J_{ j,4 }
	J_{ 2m + \ell - k - j , 4 }
	\\
	& \quad
	\leq
	\frac C { L^{ 2 ( m + 1 ) } }
	\sum_{ \alpha = 0 }^{ 2m + \ell }
	\sum_{ ( k,j ) \in K_{ \alpha , 1 } }
	\sum_{ \beta = 0 }^{ m + \ell }
	\sum_{ k^\prime + j^\prime = \beta }
	J_{ m + \ell + 1 , 2 }^{ \frac { 4 k^\prime + 1 } { 4 ( m + \ell + 1 ) } }
	J_{ 0,2 }^{ \frac { 4 ( m + \ell - k^\prime ) + 3 } { 4 ( m + \ell + 1 ) } }
	J_{ m + \ell + 1 , 2 }^{ \frac { 4 j^\prime + 1 } { 4 ( m + \ell + 1 ) } }
	J_{ 0,2 }^{ \frac { 4 ( m + \ell - j^\prime ) + 3 } { 4 ( m + \ell + 1 ) } }
	\\
	& \quad \qquad \qquad \qquad \qquad \qquad \qquad \qquad \qquad
	\times
	J_{ m + \ell + 1 , 2 }^{ \frac { 4 ( m + \ell - k^\prime - j^\prime ) + 1 } { 4 ( m + \ell + 1 ) } }
	J_{ 0,2 }^{ \frac { 4 ( k^\prime + j^\prime ) + 3 } { 4 ( m + \ell + 1 ) } }
	J_{ m + \ell + 1 , 2 }^{ \frac { 4 ( m + \ell ) + 1 } { 4 ( m + \ell + 1 ) } }
	J_{ 0,2 }^{ \frac 3 { 4 ( m + \ell + 1 ) } }
	\\
	& \quad \qquad
	+ \,
	\frac C { L^{ 2 ( m + 1 ) } }
	\sum_{ \alpha = 0 }^{ 2m + \ell }
	\sum_{ ( k,j ) \in K_{ \alpha , 2 } }
	J_{ m + \ell + 1 , 2 }^{ \frac { 4 \ell + 1 } { 4 ( m + \ell + 1 ) } }
	J_{0,2}^{ \frac { 4m + 3 } { 4 ( m + \ell + 1 ) } }
	J_{ m + \ell + 1 , 2 }^{ \frac { 4 k + 1 } { 4 ( m + \ell + 1 ) } }
	J_{0,2}^{ \frac { 4 ( m + \ell - k ) + 3 } { 4 ( m + \ell + 1 ) } }
	\\
	& \quad \qquad \qquad \qquad \qquad \qquad \qquad \qquad \qquad
	\times
	J_{ m + \ell + 1 , 2 }^{ \frac { 4 j + 1 } { 4 ( m + \ell + 1 ) } }
	J_{0,2}^{ \frac { 4 ( m + \ell - j ) + 3 } { 4 ( m + \ell + 1 ) } }
	J_{ m + \ell + 1 , 2 }^{ \frac { 4 ( 2m + \ell - k - j ) + 1 } { 4 ( m + \ell + 1 ) } }
	J_{0,2}^{ \frac { - 4 ( m - k - j ) + 3 } { 4 ( m + \ell + 1 ) } }
	\\
	& \quad
	=
	\frac C { L^{ 2 ( m + 1 ) } }
	J_{ m + \ell + 1 , 2 }^{ \frac { 2 ( m + \ell ) + 1 } { m + \ell + 1 } }
	J_{0,2}^{ \frac { 2 ( m + \ell ) + 3 } { m + \ell + 1 } }
	\\
	& \quad
	=
	\frac C { L^{ 2 ( m + 1 ) } }
	I_{ m + \ell + 1 }^{ \frac { 2 ( m + \ell ) + 1 } { 2 ( m + \ell + 1 ) } }
	I_0^{ \frac { 2 ( m + \ell ) + 3 } { 2 ( m + \ell + 1 ) } }
	\\
	& \quad
	\leq
	\frac C { L^{ 2 ( m + 1 ) } }
	\left( 
	\epsilon I_{ m + \ell + 1 } + C_\epsilon I_0^{ 2 ( m + \ell ) + 3 }
	\right).
\end{align*}
Taking
$ \epsilon > 0 $ sufficiently small, we obtain
\[
	\frac d { dt } I_\ell
	+
	\frac { 2 \ell + 1 } { L^{ 2 ( m + 1 ) } } I_m I_\ell
	+
	\frac { C_1 } { L^{ 2 ( m + 1 ) } } I_{ m + \ell + 1 }
	\leq
	\frac { C_2 } { L^{ 2 ( m + 1 ) } }
	\left(
	I_{-1}
	+
	I_0^{ \frac { 2 ( m + \ell ) + 5 } 3 }
	+
	I_0^{ 2 ( m + \ell ) + 3 }
	\right)
\]
for sufficiently large $ t > 0 $. Hence we obtain
\[
	I_\ell \leq C_\ell e^{ - \lambda_\ell t }.
\]
\qed
\end{proof}
Moreover we obtain the next theorem.
\begin{thm}
Let $ \vecf $ be as in Theorem \ref{Theorem3},
and let $ \displaystyle{ f(s,t) = \sum_{ k \in \mathbb{Z} } \hat f(k) (t) \varphi_k (s) } $ be the Fourier expansion for any fixed $ t > 0 $.
Set
\[
	\vecc (t) = \frac 1 { \sqrt { L(t) } }( \Re \hat f(0) (t) , \Im \hat f(0) (t) ) ,
\]
and define $ r(t) \geq 0 $ and $ \sigma (t) \in  \mathbb{R} / 2 \pi \mathbb{Z}  $ by
\[	
	\hat f(1) (t) = \sqrt{ L(t) } r(t) \exp \left( i \frac { 2 \pi \sigma (t) } { L(t) } \right) .
\]
Furthermore we set
\[
	\tilde { \vecf } ( \theta , t )
	=
	\vecf ( L(t) \theta - \sigma (t) , t ) ,
	\quad \mbox{for} \quad
	( \theta , t ) \in \mathbb{R} / \mathbb{Z} \times [ 0 , \infty )
	.
\]
Then the following claims hold.
\begin{itemize}
\item[{\rm (1)}]
There exists $ \vecc_\infty \in \mathbb{R}^2 $ such that
\[
	\| \vecc (t) - \vecc_\infty \| \leq C e^{ - \gamma t } .
\]
\item[{\rm (2)}]
The function $ r(t) $ converges exponentially to the constant $ \displaystyle{ \frac { L_\infty } { 2 \pi } } $ as $ t \to \infty $:
\[
	\left| r(t) - \frac { L_\infty } { 2 \pi } \right|
	\leq
	C e^{ - \gamma t } .
\]
\item[{\rm (3)}]
There exists $ \sigma_\infty \in \mathbb{R} / 2 \pi \mathbb{Z} $ such that
\[
	| \sigma (t) - \sigma_\infty | \leq C e^{ - \gamma t } .
\]
\item[{\rm (4)}]
For any $ k \in \mathbb{N} \cup \{ 0 \} $ there exist $ C_k > 0 $ and $ \gamma_k > 0 $ such that 
\[
	\| \tilde { \vecf } ( \cdot , t ) - \tilde { \vecf }_\infty \|_{ C^k ( \mathbb{R} / \mathbb{Z} ) }
	\leq
	C_k e^{ - \gamma_k t }
	,
\]
where
\[
	\tilde { \vecf }_\infty ( \theta )
	=
	\vecc_\infty + \frac { L_\infty } { 2 \pi } ( \cos 2 \pi \theta , \sin 2 \pi \theta ) .
\]
\item[{\rm (5)}]
For sufficiently large $ t $,
$ \mathrm{Im} \tilde { \vecf }( \cdot , t ) $ is the boundary of a bounded domain $ \Omega (t) $.
Furthermore, there exists $ T_\ast \geq 0 $ such that $ \Omega (t) $ is strictly convex for $ t \geq T_\ast $.
\item[{\rm (6)}]
Let $ D_{ r_\infty } ( \vecc_\infty ) $ be the closed disk with center $ \vecc_\infty $ and radius $ r_\infty $.
Then we have 
\[
	d_H ( \overline { \Omega(t) } , D_{ r_\infty } ( \vecc_\infty ) )
	\leq
	C e^{ - \gamma t } ,
\]
where $ d_H $ is the Hausdorff distance.
\item[{\rm (7)}]
Let $\displaystyle 
	\vecb (t) =
	\frac 1 { A(t) } \iint_{ \Omega (t) } \vecx \, d \vecx$
be the barycenter of $\Omega(t)$. Then we have
\[
\|A(t) ( \vecb (t) - \vecc (t) )\| \leq C e^{ - \gamma t }.
\]
\end{itemize}
\label{Theorem4}
\end{thm}
\begin{proof}
From Theorem \ref{Theorem3}, we have
\[
\|\tilde{\kappa}\|_{C^{\infty}} \leq Ce^{-\lambda t}.
\]
Hence we can show each of the above assertions in the same way as in \cite[Theorem 4.3]{NN}.
\qed
\end{proof}
\par\noindent
{\bf Acknowledgment}.
The author expresses their appreciation to Professor Takeyuki Nagasawa for discussions.
The author also would like to express their gratitude to Professor Neal Bez for English language editing.


\begin{thebibliography}{99}
\bibitem{C}
	Chou, K.-S.,
 	{\it A blow-up criterion for the curve shortening flow by surface diffusion},
	Hokkaido Math. J.,\ {\bf 32} (2003),
	1--19.
\bibitem{DKS}
	Dziuk, G., E. Kuwert \& R. Sch\"{a}tzle,
 	{\it Evolution of elastic curves in $\Bbb R\sp n$:
	existence and computation},
	SIAM J. Math.\ Anal.\ {\bf 33} (5) (2002),
	1228--1245.
\bibitem{EG}
	Elliott, C.M. \& H. Garcke,
	{\it Existence results for diffusive surface motion laws},
	Adv.\ Math.\ Sci.\ Appl.\ {\bf 7} (1) (1997),
	467--490. 
\bibitem{EI}
	Escher, J. \& K. Ito,
	{\it Some dynamic properties of volume preserving curvature driven flows},
	Math.\ Ann.\ {\bf 333} (1) (2005),
	213--230. 
 \bibitem{G}
	Gage, M.,
	{\it On an area-preserving evolution equation for plane curves},
	in ``Nonlinear problems in geometry (Mobile, Ala., 1985)'',
	Contemp.\ Math.\ {\bf 51},
	Amer. Math. Soc.,
	Providence,
	1986,
	pp.51--62.
 \bibitem{GI}
         Giga, Y. \& K. Ito,
         \textit{Loss of convexity of simple closed curves moved by surface diffusion}, 
         in Topics in nonlinear analysis, Progr.\ Nonkinear Differential Equations Appl., {\bf 35} 
         Birkh\"{a}user, Basel, (1999), 305--320.
\bibitem{Mu}
        Mullins, W. W., 
        \textit{Two-dimensional motion of idealized grain boundaries}, 
        J.\ Appl. Phys. {\bf 27} (1956), 900--904.
\bibitem{M}
	Mayer, U. F.,
	{\it A singular example for the averaged mean curvature flow},
	Experiment.\ Math.\ {\bf 10} (1) (2001),
	103--107.
\bibitem{NN} 
        Nagasawa, T. \& K. Nakamura,
        {\it Interpolation inequalities between the deviation of curvature and the isoperimetric ratio with applications to geometric flows},
        submitted, arXiv:1811.10164.
\bibitem{W}
	Wheeler, G.,
	{\it On the curve diffusion flow of closed plane curves},
	Annali di Matematica\ {\bf 192} (2013),
	931--950.
\end{thebibliography}
\end{document}